\documentclass[12pt, reqno]{amsart}
\usepackage{amsmath, amsthm, amscd, amsfonts, amssymb, graphicx, color}
\usepackage{lscape}
\usepackage[all]{xy}

\def\C{\mathbb{C}}

\def\N{\mathbb{N}}

\def\Z{\mathbb{Z}}
\def\Q{\mathbb{Q}} 

\DeclareMathOperator{\lcm}{lcm}
\DeclareMathOperator{\supp}{supp}
\DeclareMathOperator{\ord}{ord}
\newtheorem{Th}{Theorem}[section]
\newtheorem{Lm}[Th]{Lemma}
\newtheorem{Pro}[Th]{Proposition}
\newtheorem{Co}[Th]{Corollary}
\newtheorem{Ep}[Th]{Example}
\newtheorem{rk}[Th]{Remark}
\newtheorem{qu}[Th]{Question}

\numberwithin{equation}{section}

\begin{document}

\setcounter{page}{1}

\title{On finite factorization Puiseux algebras}

\author{M. Benelmekki}

\address{Department of Mathematics, Facult\'e des Sciences et Techniques, 
Beni Mellal University, P.O. Box 523, Beni Mellal, Morocco}
\email{ med.benelmekki@gmail.com}
%\email{Said El Baghdadi:  s.elbaghdadi@usms.ma}

\date{\today}

\begin{abstract}

An integral domain $D$ is called a finite factorization domain (FFD) if every nonzero nonunit element of $D$ has only  finitely many non-associate divisors. In 1998, for an integral
domain $D$ and a  cancellative torsion-free monoid $S$ such that each nonzero
element of its quotient group is of type $(0,0, \ldots)$,  Kim proved that the monoid domain $D[S]$ is an FFD if and only if $D$ is an FFD and $S$ is an FFM. However,  it is still open whether a monoid algebra $K[S]$ is an FFD provided that $S$ is a reduced FFM. In this paper, we show that a Puiseux algebra $K[S]$ is an FFD if and only if $S$ is an FFM, when $K$ is a finitely generated field of characteristic $0$. This would provide a large  class of one-dimensional monoid algebras with finite factorization property. We also prove that every generalized cyclotomic polynomial has the finite factorization property in $K[S]$ where $S$ is a reduced FFM and $K$ is an arbitrary field of characteristic $0$.  
\end{abstract}

\maketitle

\section{Introduction}

 Let $D$ be an integral domain and let $S$ be a commutative cancellative torsion-free monoid, written additively. Denote by $\langle S\rangle$ the quotient group of
$S$ and $(S, <)$ a total order on $S$. The monoid domain of $S$ over $D$ is defined by $$D[S]=\lbrace \sum_i a_iX^{s_i}|\,a_i\in D \mbox{ and } s_i\in S\rbrace.$$
 Every nonzero element $f\in D[S]$ has a unique representation in the~form $$ f= a_{1}X^{s_1}+a_{2}X^{s_2}+\cdots+a_{n}X^{s_n},$$ where $n\in \mathbb{Z}_+$, $a_{i}\in D\setminus \{0\} $ and ${s_i}\in S$ $(i=1,...,n)$ such that $s_1<s_2<\cdots<s_n$. The subset $\supp f:=\lbrace s_1,\ldots,s_n \rbrace$ of $S$ is called the support of $f$. We let $\ord f:=s_1$ and $\deg f:=s_n$ denote the order and degree of $f$, respectively.  These domains emerge as a significant generalization of classical polynomial rings: if $S=\mathbb{Z}_+^r$ for some positive integer $r$, then $D[S]=D[X_1,...,X_r]$
 and $D[\langle S\rangle]=D[X_1^{\pm 1},...,X_r^{\pm 1}]$. Academic interest in the broader field of commutative semigroup rings, which notably comprises monoid domains, expanded significantly from the 1970s.  This increased interest was largely thanks to the important early work of mathematicians like Gilmer and Matsuda. Gilmer's significant books, especially his comprehensive volume on commutative semigroup rings \cite{G84}, truly encouraged a lot of new research into  the algebraic and factorization properties of these domains, motivated by their role as natural generalizations of polynomial rings and their profound connections to areas like algebraic geometry and homological algebra.

Let $S$ be a multiplicative monoid and let $U(S)$ denote the set of units of $S$. If $U(S)$ is trivial, then $S$ is said to be reduced. For $s,t\in S$,  we say that $s$ divides $t$ if there exists $r\in S$ such that $t=rs$. If $t=us$ for some $u\in U(S)$, we say that $s$ and $t$  are associate. By an irreducible element (or atom) of $S$ we mean a nonunit $s\in S$ such that $s=tr$, where  $t,r\in S$, implies that $t$ or $r$ is a unit of $S$. We say that $S$ is atomic if each nonunit of $S$ is a product of a finite number of irreducible elements of $S$. The  monoid $S$ is a bounded factorization monoid (BFM) if $S$ is atomic and for each nonzero nonunit element of $S$ there is a bound on the length of its factorizations into products of irreducible elements. We say that $S$ is an idf-monoid if each nonzero nonunit element of $S$ has only a finite number of non associate irreducible divisors. An atomic idf-monoid $S$ is called a finite factorization monoid (FFM); equivalently, each nonzero nonunit element of $S$ has only a finite number of nonassociate divisors in $S$. For an integral domain $D$, the set of units is denoted by  $U(D)$. If the multiplicative monoid $D\setminus \{0\}$, of nonzero elements of $D$,  is atomic (resp., BFM, FFM, idf-monoid), $D$ is called atomic (resp., BFD, FFD, idf-domain). The concepts of BFD and FFD are due to Anderson, Anderson, and Zafrullah \cite{AA90}, while the class of idf-domains was introduced by Grams and Warner in \cite{GR75}. For more on these factorization properties see \cite{AA90,AA92}. A survey on the most relevant results about bounded and finite factorization domains can be found in \cite{AG20}.

 In factorization theory, a classical problem in the study of monoid domains concerns determining the conditions under which the monoid domain $D[S]$ satisfies a certain factorization property, a question that continues to be intensively researched (see for instance, \cite{BE25,BEN,Got22,GL,GR25}). For an integral domain $D$ and a cancellative torsion-free  monoid $S$ such that each nonzero element of its quotient group is of type $(0,0, \ldots)$, Kim proved that $D[S]$ is an FFD (resp., a BFD) if and only if $D$ is an FFD (resp., a BFD) and $S$ is an FFM (resp., a BFM), see \cite[Propositions 3.15 and 3.25]{Kim98}. For a general  cancellative torsion-free  monoid $S$, 
 Anderson and Juettin, in \cite{AJ}, proved that that $D[S]$ is a BFD  if and only if  $D$ is an BFD and $S$ is an BFM, where they assume that $S$ is reduced,  without requiring that every nonzero element of its quotient group is of type $(0,0,\ldots)$. This question is considered also by Gotti for the case of monoid domains that consist of polynomial expressions with non-negative rational exponents over a field, called Puiseux algebras. He proved that, for a field $K$ and a strongly increasing monoid $M$ (i.e.,   its elements are the terms of an increasing sequence of rationals), the Puiseux algebra $K[M]$ is an FFD, see \cite[Proposition 4.10]{Got22}. However, a characterization remains unknown even for the class of Puiseux algebras, see \cite[Question 4.9]{Got22}.  
 
 The main purpose of this paper is to investigate the ascent of the finite factorization property (FF-property) to Puiseux algebras. In section 2, we define elements with symmetric support in a monoid domain and subsequently apply their properties to show that an important class of elements of $K[S]$ has the FF-property.  More precisely, for $0\ne s\in S$ and a polynomial $\Phi$ whose zeros are roots of unity over $K$, we show that the generalized cyclotomic polynomial $\Phi(X^s)$ satisfies the FF-property in $K[S]$, where $S$ is an arbitrary reduced FFM and $K$ is an arbitrary field  of characteristic $0$. In section 3, we prove that the Puiseux algebra $K[S]$ is an FFD if and only if $S$ is an FFM, where $K$ is a finitely generated field of characteristic $0$. This work will establish a significant class of one-dimensional monoid algebras possessing the FF-property. Moreover, this contribution can also be viewed as an initial step towards answering the following general question. 
 \begin{qu}
 Let $K$ be any field and let $S$ be a torsion-free cancellative monoid. Is the monoid algebra $K[S]$ an FFD provided that $S$ is a reduced FFM?
\end{qu}
  Throughout, $\mathbb{N}$, $\mathbb{Z}$, and $\mathbb{Q}$ denote the positive integers, the integers, and the rational numbers, respectively. The sets $\mathbb{Z}_+$ 
 and  $\mathbb{Q}_+$ will denote the monoid of non-negative, integers and rational numbers, respectively.  All monoids considered in this paper will be assumed to be commutative.  For $r=\frac{n}{d}\in \mathbb{Q}_+$, 
$(n,d)\in\mathbb{Z}_+\times\mathbb{N}$, reduced (i.e., $\gcd(n,d)=1$), we call the unique pair of integers $n$ and $d$, the numerator and the denominator of $r$ and denote them by $n(r)$ and $d(r)$, respectively. For a nonempty subset $A$ of $\mathbb{Q}_+$, we set $d(A)=\lbrace d(r) : r\in A\rbrace$; that is the set of the denominators of all (reduced) elements of $A$. All exponents  that occur in the representation of an element $f\in K[\mathbb{Q}_+]$, considered in this paper, will be assumed to be reduced.  
General references for any undefined terminology or notation are \cite{G72, G84, Sch}.

\section{Factorization of elements with symmetric support}
In this section, we show that any product of cyclotomic polynomials over a field of characteristic $0$, when considered as an element of a monoid algebra, satisfies the finite factorization property. Our approach is primarily guided by the properties of elements with symmetric support.

Let $D$ be an integral domain and let $S$ be a cancellative torsion-free monoid, written additively. We say that an element $f\in D[S]$ is of \textit{symmetric support} if for every $s\in \supp f$, $(\deg f+\ord f-s)\in \supp f$. In particular, monomials $aX^s$, where $0\neq a\in D$ and $ s\in S$, are of symmetric support in $D[S]$. Note that palindromic and antipalindromic polynomial in $D[X]$ are of symmetric support. 
 
 It is well known that the product of two palindromic or antipalindromic polynomials, with coefficients from an arbitrary field $K$,  is palindromic. Also, the product of a palindromic polynomial and an antipalindromic polynomial is antipalindromic. Thus, all such products are of symmetric support. Generally, the product of two elements of symmetric support is not so. For instance, consider the elements $f(X)=X+1$ and  $g(X)=X^2-X+2$ of $\Q[X]$. Clearly, $f$ and $g$ are of symmetric support in $\Q[X]$. However, their product $f(X)g(X)=X^3+X+2$ is not of symmetric support. 
 
 It is important to note that there are some special elements with symmetric support such that their product is also of symmetric support. For example, consider a cyclotomic polynomial $\Phi$ over $\Q$. Then $\Phi$ is of symmetric support since every cyclotomic polynomial over $\Q$ is either palindromic or antipalindromic. Note that $X-1$ is the only antipalindromic cyclotomic polynomial over $\Q$.  Therefore, any finite product of cyclotomic polynomials over $\Q$ is of symmetric support.
 
Our next objective is to extend the previous statement to a large class of fields.  We use the following terminology. Let $K$ be a field and let $K[X_1,\ldots,X_n]$ be the polynomial ring in $n$ variables over $K$. For $k \in \Z_+$, the $k$-th elementary symmetric polynomial, often denoted $e_k$,  defined as follows:
    $$e_k(X_1, \ldots, X_n) = \sum_{1 \le i_1 < i_2 < \ldots < i_k \le n} X_{i_1} X_{i_2} \ldots X_{i_k}.$$
By convention, $e_0 = 1$. One of the most important applications of elementary symmetric polynomials is their direct connection to the coefficients of a monic polynomial. If a monic polynomial $P(X)$ of degree $n$ has roots $a_1, a_2, \ldots, a_n$, then there is a relationship between its coefficients  and  the elementary symmetric polynomials of its roots as follows:
$$P(X) = \prod_{i=1}^n (X-a_i) =X^n +\sum_{k=0}^{n-1} (-1)^k e_ {n-k}(a_1, \ldots, a_n) X^{k}.$$
This identity is known as \textit{Vieta's Formulas}. When each $a_i\neq 0$, it is well known that for every $k=0,\ldots,n$,
 $$e_k(a_1, \ldots, a_n) = e_n(a_1, \ldots, a_n). e_{n-k}(a_1^{-1}, \ldots, a_n^{-1}).$$
In the next result, we will use of the following elementary lemma that is of interest in itself.
\begin{Lm}\label{L1}
Let $K$ be a field of characteristic $0$. Let $a_1, a_2, \ldots, a_n$ be $n$ roots of unity  over $K$, not necessarily distinct. If $e_k(a_1, \ldots, a_n) = 0$ for some $0 \le k \le n$, then $e_{n-k}(a_1, \ldots, a_n) = 0$.
\end{Lm}

\begin{proof}
Let $a_1, a_2, \ldots, a_n$ be $n$ roots of unity in an algebraic closure $\overline{K}$ of $K$. Since $K$ has characteristic 0, we can embed its algebraic closure $\overline{K}$ into the field of complex numbers $\C$. Thus, we can consider these roots to be complex numbers.

Let  $e_j = e_j(a_1, \ldots, a_n)$, for every $j=0,\ldots,n$. Since each $a_i$ is a root of unity, $a_i \ne 0$ for every $i=1,\ldots,n$. Consequently, their product $e_n = \prod_{i=1}^n a_i$ is also nonzero.
 Let $\overline{a_i}$ denotes the complex conjugate of $a_i$. Since $|a_i|=1$, $a_i \overline{a_i} = |a_i|^2 = 1$, and so $a_i^{-1} = \overline{a_i}$ for all $i=1, \ldots, n$. On the other hand, we have 
 $$e_k(a_1, \ldots, a_n) = e_n(a_1, \ldots, a_n). e_{n-k}(a_1^{-1}, \ldots, a_n^{-1}).$$
 Thus 
$$ e_{k} = e_n.\overline{e_{n-k}}, $$
where $\overline{e_j} = e_j(\bar{a_1}, \ldots, \bar{a_n}) = \overline{e_j(a_1, \ldots, a_n)}$, for every $j=0,\ldots,n$. Therefore, if $e_k = 0$, then $\overline{e_{n-k}}=0$ since $e_n \neq 0$. Consequently,  $e_{n-k} = 0$. 
\end{proof}
 Following \cite{Sch}, a cyclotomic polynomial $\Phi$ over a field $K$ is a monic univariable polynomial irreducible over $K$ whose zeros are roots of unity (i.e., $\Phi$ is an irreducible divisor of $X^\delta-1$ in $K[X]$ for some positive integer $\delta$). 
\begin{Pro} \label{p1}
Let $K$ be a field of characteristic $0$.  Then any finite product of cyclotomic polynomials is of symmetric support in $K[X]$. In particular, every cyclotomic polynomial over $K$ is of symmetric support.
\end{Pro}
\begin{proof}
Consider cyclotomic polynomials $\Phi_1,\ldots\Phi_t\in K[X]$, and 
let $a_1, a_2, \ldots, a_n$ the multiset (elements can appear multiple times) of all the roots (of unity) of the polynomial $\Phi_i$ over $K$. Let $\Phi(X)=\prod_{i=1}^t \Phi_i(X)$. Note that $\Phi$ is a monic polynomial, of degree $\deg\Phi= n$ and $\ord \Phi =0$,  that has  $a_1, a_2, \ldots, a_n$ as roots. Then, by Vieta's formulas, $\Phi$ can be written as:
$$\Phi(X) = X^n - e_1 X^{n-1} + e_2 X^{n-2} - \cdots + (-1)^k e_k X^{n-k} + \cdots + (-1)^n e_n,$$
where $e_j = e_j(a_1, \ldots, a_n)$. Assume now that $s\in \supp \Phi$; that is, $s=n-k$ for some $0 \le k \le n$ and $e_{k}\neq 0$. Then, it follows from Lemma \ref{L1} that $e_{n-k}\neq 0$. Thus, $k\in \supp \Phi$, and so $n-s\in \supp \Phi$.
\end{proof}

The next example shows that Proposition \ref{p1} does not holds, in general, for a field of positive characteristic.
\begin{Ep}
  Let $K = \mathbb{F}_2$, the finite field with two elements. Consider the polynomial $\Phi(X) = X^3+X^2+1$ over $\mathbb{F}_2$. This polynomial is irreducible over $\mathbb{F}_2$. Its roots, say $a_1, a_2, a_3$, are nonzero elements of $\mathbb{F}_2[X]/(X^3+X^2+1)$, which is isomorphic to $\mathbb{F}_8$.  Consequently, they are 7th roots of unity, and therefore $\Phi$ is a cyclotomic polynomial over $K$. However, $\Phi$ is not of symmetric support in $K[X]$. Note that $e_1(a_1, a_2, a_3)=1$ but $e_{3-1}(a_1, a_2, a_3)=0$. This also shows  that, for a field of positive characteristic, Lemma \ref{L1} does not holds in general.    
 \end{Ep} 
 
Let $D$ be an integral domain and let $S$ be a cancellative torsion-free  monoid. Following \cite{BEN}, an element  of the form $f(X^s)\in D[S]$, where $s\in S$ and $f(X)\in D[X]$, is called a \textit{generalized polynomial}. These elements play an important role in factorization theory in monoid domains, see \cite{BEN}.  
\begin{rk}\label{r2}
Let $D$ be an integral domain and let $S$ be a cancellative torsion-free  monoid. For $s\in S$ and  $f(X)\in D[X]$ of symmetric support, it is easy to see that the generalized polynomial $f(X^s)$ is of symmetric support in $D[S]$. In particular,  for every non-negative rational number $r$, $f(X^r)$ is  of symmetric support in $D[\Q_+]$. 
\end{rk}

 Let $D$ be an integral domain. Recall that a nonzero nonunit element of $D$ satisfies the \textit{FF-property} in $D$ if it has only a finite number of factorizations in $D$, up to associates. 
 
We now have the results needed to prove the main result of this section. The following theorem  prove that the class of generalized cyclotomic polynomials in $K[S]$ has the FF-property in $K[S]$.
\begin{Th}\label{T3}
Let $K$ be a field of characteristic $0$ and let $S$ be a cancellative torsion-free reduced monoid.  Let $0\neq s\in S$ and let $\Phi \in K[X]$ be a product of cyclotomic polynomials  over $K$. If $S$ is an FFM, then $\Phi(X^s)$ satisfies the FF-property in  $K[S]$. 
\end{Th}
\begin{proof}
 Let $0\neq s\in S$ and let $\Phi \in K[X]$ be a product of cyclotomic polynomials  over $K$. Without loss of generality, we may suppose that $s>0$. Assume that there exists an infinite set $\{\phi_i\}_{i\in I}$ of nonassociate divisors of $\Phi(X^s)$ in $K[S]$. Then, for every $i\in I$, there exists $\psi_i\in K[S]$ such that 
$$\Phi(X^s)=\phi_i(X)\psi_i(X).$$
It should be noted that $\phi_i$ and $\psi_i$ are not monomials and $\ord \phi_i=\ord \psi_i=0$.  Since $S$ is an FFM, we may suppose that each $\phi_i$ (resp., $\psi_i$) have the same  degree, say $d_1$ (resp., $d_2$). Also,  note that $\phi_i$ and $\psi_i$ are both of symmetric support in $K[S]$, for every $i\in I$. To see this, fix $i\in I$ and let $G$ be the quotient group of $S$.  By \cite[Proposition 3.3]{BEN}, there exist a positive integer $n_i$ and $\alpha_i\in G$ such that $s=n_i\alpha_i$, $\phi_i=\Phi_i(X^{\alpha_i})$, and $\psi_i=\Psi_i(X^{\alpha_i})$, where $\Phi_i,\Psi_i \in K[X]$ with $\Phi_i(0) \neq 0$.  
  Then, $\Phi(X^{n_i})=\Phi_i(X)\Psi_i(X)$. 
Hence, each root of either $\Phi_i$ or $\Psi_i$ is  a root of unity over $K$. Thus, it follows from  Proposition \ref{p1} and Remark \ref{r2} that $\phi_i$ and $\psi_i$ are of symmetric support in $K[G]$, and hence in $K[S]$.  This implies that there exists a finite subset $A$ of $S$ such that,  for every $i\in I$, 
 $$\supp \phi_i \cup \supp \psi_i \subseteq A.$$  Indeed, let $i\in I$, $s_i\in \supp \phi_i$, and  $t_i\in \supp \psi_i$. Then, $d_1-s_i\in \supp \phi_i$.  Since $\supp \phi_i\subset S$ and $d_1=(d_1-s_i)+s_i$,   $s_i$ divides  $d_1$ in $S$. By a similar argument, $t_i$ divides  $d_2$ in $S$. Therefore, the set $$\bigcup_{i\in I}(\supp \phi_i \cup \supp \psi_i)$$ consists of  elements of $S$ that divide $d_1+d_2$ in $S$. Since $S$ is a reduced FFM, it must be finite.

Consider now the monoid $M$ generated by the set $A$, and let $K[M]$ the monoid algebra of $M$ over $K$. Clearly $M$ is finitely generated and $U(K[M])=U(K[S])$. As a consequence, $\{\phi_i\}_{i\in I}$ contains infinitely many nonassociate divisors of $\Phi$ in $K[M]$. This is a contradiction since $K[M]$ is an FFD  \cite[Corollary 2.29]{AG20}.
\end{proof}
As a consequence of Theorem \ref{T3}, we have the following corollary.
 \begin{Co}\label{C3}
Let $K$ be a field of characteristic $0$ and let $S$ be a Puiseux monoid containing $1$. Assume that $S$ is an FFM. Then any product $\Phi \in K[X]$ of cyclotomic polynomials satisfies the FF-property in  $K[S]$. 
\end{Co}

\section{Finite factorization Puiseux algebras}

Let $D$ be an integral domain and let $S$ be a cancellative torsion-free monoid such that each nonzero element of its quotient group is of type $(0,0, \ldots)$. Kim proved that the monoid domain $D[S]$ is an FFD if and only if $D$ is an FFD and $S$ is an FFM. This is remains valid if the term “BF” is substituted for “FF”, see \cite[Propositions 3.15 and 3.25]{Kim98}. For a  cancellative torsion-free monoid $S$, 
 Anderson and Juettin, in \cite{AJ}, proved that that $D[S]$ is a BFD  if and only if  $D$ is a BFD and $S$ is a BFM,  where they assume that $S$ is reduced, but without assuming that each nonzero element of its quotient group  is of type $(0,0,\ldots)$. However, this question is still open for FFD. Even for the class of Puiseux algebras, no such characterization is known, see \cite[Question 4.9]{Got22}.  The primary goal of this section is to address this question.

\begin{rk}\label{r1}
For an additive cancellative torsion-free monoid $S$ and a nonzero rational number $r$, the monoid domain $D[rS]$ is isomorphic to $D[S]$, where $D$ is an integral domain. This is because the mapping $\sigma_r: S \to rS$ defined by $\sigma_r(s) = rs$ is a monoid isomorphism. Therefore,  $\sigma_r$ give rise to an $D-$algebras isomorphism $\Psi_r: D[S] \to D[rS]$ defined by, for every $f = \sum_{i=1}^k d_i X^{s_i} \in D[S]$, 
$$\Psi_r(f) := \sum_{i=1}^k d_i X^{\sigma_r(s_i)} = \sum_{i=1}^k d_i X^{rs_i}.$$
 Consequently,  $D[S]$ and $D[rS]$ share all ring-theoretic properties. In particular, $D[S]$ is an FFD if and only if $D[rS]$ is an FFD.
\end{rk}

We are now in a position to prove the main result of this paper.  
\begin{Th}\label{L33}
Let $K$ be a finitely generated field of characteristic $0$ and let $S$ be a Puiseux monoid.  Then $K[S]$ is an FFD if and only if  $S$ is an FFM.
\end{Th}
\begin{proof}
The direct implication is obvious. To prove the reverse implication, suppose that $S$ is an FFM. By virtue of Remark \ref{r1}, we may assume that $1\in S$. 

First, let us prove that every $f\in K[X]$ satisfies the FF-property in $K[S]$. If $f$ is a monomial then the result is immediate since $S$ is an FFM.  Thus,  let $f$ be a nonzero nonunit element of $K[X]$ that is not a monomial, and let $\{f_i\}_{i\in I}$ be an infinite set of nonassociate divisors of $f$ in $K[S]$. Since $S$ is an FFM, we may suppose that each $f_i$ is not a monomial and that they have the same order and degree, say $s$ and $d$, respectively. Then, for every $i\in I$, there exists $h_i\in K[S]$ such that 
\begin{equation}\label{E1}
f(X)=f_i(X)h_i(X).
\end{equation}
 By   \cite[Theorem 3.11]{BE22}, $f$ has a factorization in $K[\Q_+]$ of the following form: 
 \begin{equation}\label{E2}
 f=cX^{r}\left(\prod\limits_{j=1}^{\lambda}\Phi_{j}^{e_j}(X)\right)\left( \prod\limits_{j=1}^{\rho} \varphi_j^{l_j}\right),
 \end{equation}
where \begin{itemize}
\item  $r=\ord f \in \mathbb{Z}_+$ and $c\in K$.
\item $\lambda,t, e_j,l_j\in \mathbb{Z}_+$.
\item $\Phi_{j}\in K[X]$ a cyclotomic polynomial that has no irreducible divisors in $K[\mathbb{Q}_+]$.
\item The $\varphi_j$'s are irreducible (prime) in $K[\mathbb{Q}_+]$ and uniquely determined, up to associates.
\end{itemize} 
Let $m=\lcm \lbrace d\left(  \left(\bigcup _{j=1}^{t}\supp \varphi_j\right) \cup\{s,d\} \right) \rbrace $. Then (\ref{E2}) implies that
$$f(X^m)=c\Phi(X) g(X),$$
where 
\begin{itemize}
\item  $\Phi(X):=\prod\limits_{j=1}^{\lambda}\Phi_{j}^{e_j}(X^m) \in K[X]$. 
\item $g(X):=X^{rm}\prod\limits_{j=1}^{\tau} g_j^{l_j}(X)$ with $g_j(X)=\varphi_j(X^m)\in K[X]$ and $\gcd(g_j,X^\delta-1)=1$, for every positive integer $\delta$. Note that the $g_i$'s are primes in  $K[\mathbb{Q}_+]$.  
\end{itemize} 
Then (\ref{E1}) become $f(X^m)=f_i(X^m)h_i(X^m) $. Since $g_1,...,g_\tau$ are the only irreducible (prime) divisors of $f(X^m)$ of $f$ in $K[\Q_+]$, there exist an infinite subset $J$ of $I$ and   $l_{0,1}, \dots, l_{0,t} \in  \Z_+$ such that  $\prod\limits_{j=1}^{\tau} g_j^{l_{0,j} }(X)$ divides $f_i(X^m)$ in $K[\Q_+]$,  for every $i \in J$. By virtue of \cite[Theorem 3.11]{BE22}, for every $i\in J$, we have 
\begin{equation}\label{E3}
f_i(X^m)=c_i \phi_i(X)g_0(X),
\end{equation}
where $c_i\in K$, $\phi_i \in K[\Q_+] $, and $g_0(X):= X^{ms}\prod\limits_{j=1}^{\tau} g_j^{l_{0,j} }(X)\in K[X]$. Note that each $\phi_i$ has order $0$ and is of symmetric support since it is a finite product of element of the form $\Phi_i(X^\frac{1}{\mu_i})$ where $\mu_i$ is a positive integer and $\Phi_i\in K[X]$ is a product of cyclotomic polynomials, see Proposition \ref{p1} and Remark \ref{r2}. Also, note that there is infinitely many indices $i$ such that $\phi_i\neq 1$ and $ \supp \phi_i\not\subset \Z_+$. Moreover, the $\phi_i$'s have the same degree, say  $d'\in \N$. Now, we have two cases:

\textbf{Case 1}: $g_0\neq 1$. 

 Set $g_0=\sum_{j=1}^N a_j X^{n_j}$ where $a_j\in K\setminus\{0\}$, and $n_j\in \Z_+$ such that $n_1<n_2<\cdots<n_N$. For $i\in J$, we let $\gamma(i)$ denotes the smallest element of $(\supp \phi_{i})\setminus \Z_+$. Since $\phi_i$ is of symmetric support in  $K[\mathbb{Q}_+]$, $d'-\gamma(i) \in \supp \phi_i$. Note that $\gamma(i),(d'-\gamma(i)) \notin \Z_+$. Using the fact that $\supp g_0  \subset \Z_+$ and that $\{t\in \supp \phi_i | t < \gamma(i) \mbox{ or } t> d'-\gamma(i) \} \subset \Z_+ $,  it follows from the equation (\ref{E3}) that $(n_1+\gamma(i)),(n_N +d'-\gamma(i))\in \supp f_i(X^m)$. Then  $$n_1+n_N +d'=(n_1+\gamma(i))+(n_N +d'-\gamma(i)).$$ 
In particular, since $\supp f_i(X^m) \subset S$,  $n_1+\gamma(i)$ is a divisor of $n_1+n_N +d'$ in $S$ for every $i\in J$. Since $S$ is an FFM, there exist $\gamma_{1}=\dfrac{p}{q}\in \Q_+ \setminus \Z_+$ and an infinite subset $J_1$ of $J$ such that, for every $i\in J_1$, $\gamma(i)=\gamma_{1}$. Moreover, $$n_1+n_N +d'=(n_1+\gamma_1)+(n_N +d'-\gamma_1),$$ and hence $n_1+\gamma_1$ divides $n_1+n_N +d'$ in $S$. 
 
 Now, consider the substitution $X\rightarrow X^q$ in the equation (\ref{E3}). Thus, for every every $i\in J_1$, we have  $$f_i(X^{mq})=c_i \phi_i(X^q)g_0(X^q).$$
 Again, for $i\in J_1$, let $\gamma(i)$ be the smallest element of $(\supp \phi_{i}(X^q))\setminus \Z_+$. By using the same argument as above, we get $\gamma_{2}'\in \Q_+ \setminus \Z_+$ and an infinite subset $J_2$ of $J_1$ such that, for every $i\in J_2$, $\gamma(i)=\gamma_{2}'$. In addition, $$qn_1+qn_N +qd'=(qn_1+\gamma_2')+(qn_N +qd'-\gamma_2'),$$ 
 where $(qn_1+\gamma_2'),(n_N +d'-\gamma_2')\in \supp f_i(X^{qm})$. Hence $$n_1+n_N +d'=(n_1+\gamma_2)+(n_N +d'-\gamma_2),$$
  with $\gamma_{2}=q^{-1}\gamma_{2}'$. Since $\supp f_i(X^{qm})=\{qt : t\in\supp f_i(X^{m})\}$, it follows that $(n_1+\gamma_2),(n_N +d'-\gamma_2)\in \supp f_i(X^{m})\subset S$. That is,  $n_1+\gamma_2$ divides $n_1+n_N +d'$ in $S$. Note that $\gamma_1<\gamma_2$. 
 
Continuing in this manner, we get $\gamma_1,\ldots,\gamma_k \in \Q_+ \setminus \Z_+$, $\gamma_1<\cdots<\gamma_k$, and an infinite subset $J_k$ of $I$ such that, for every $i\in J_k$ and every $j=1,\ldots,k$, $\gamma_j\in \supp \phi_i$ and $n_1+\gamma_j\in \supp f_i(X^{m})$  divides $n_1+n_N +d'$ in $S$. Since $S$ is FFM and $n_1+\gamma_1<\cdots<n_1+\gamma_k$, $k$ is bounded by the (finite) number of divisors of $n_1+n_N +d'$ in $S$. Thus,  consider $k$ to be maximal. Then, for every $i\in J_k$, $\lcm\lbrace d(\supp \phi_{i}\setminus \Z_+)\rbrace=\lcm\lbrace d(\{ \gamma_1,\ldots,\gamma_k \})\rbrace$. Since $|\supp \phi_{i}\cap\, \Z_+|\leq d'+1$ for every $i\in J_k$, it follows that there exists infinitely many $\phi_{i}$'s with the same support. Hence,  there exists an infinite subset $J_0$ of $I$ such that the elements of $\{f_i\}_{i\in J_0}$ have the same support, say $A$. Set $m'=\lcm \lbrace d(A)  \rbrace $. Then, for every $i\in J_0$, the equation (\ref{E1}) implies that
$$f(X^{m'})=f_i(X^{m'})h_i(X^{m'}).$$
Since $f(X^{m'}), f_i(X^{m'})\in K[X]$, $h_i(X^{m'})$ is also an element of $K[X]$. But $K[X]$ is UFD, and therefore there exist $i_1,i_2\in J_0$ with $i_1\neq i_2$ such that $f_{i_1}(X^{m'})$ and $f_{i_2}(X^{m'})$ are assocaite in $K[X]$; that is, $f_{i_1}(X^{m'})=u f_{i_2}(X^{m'})$ for some $u\in K\setminus \{0\}$. Thus, $f_{i_1}= u f_{i_2}$ which is a contradiction. Consequently, $f$ has a finite number of nonassociate divisors in $K[S]$. 

\textbf{Case 2:} $g_0= 1$. 

In this case, we may suppose that $g=1$. Otherwise,  there will be an infinite subset $J$ of $I$ and $l_{1,1}, \dots, l_{1,t} \in  \Z_+$ such that  $\prod\limits_{j=1}^{\tau} g_j^{l_{1,j} }(X)$ divides $h_i(X^m)$ in $K[\Q_+]$ for every $i \in J$. Hence $g_1(X):=X^{mt}\prod\limits_{j=1}^{\tau} g_j^{l_{1,j} }(X)$, where $t=\ord h_i$,  divides $h_i(X^m)$ in  $K[\Q_+]$.  Since $g\neq 1$,  $g_1\neq 1$, and hence this leads us back to Case 1 (apply the same argument, as in Case 1, to the set $\{h_i\}_{i\in J}$). Therefore, $f$ is associate to a product of cyclotomic polynomials over $K$.  Thus,  $f(X)=c\prod\limits_{j=1}^{\lambda}\Phi_{j}^{e_j}(X)$. By Corollary \ref{C3}, $f$ satisfies the FF-property in $K[S]$. This completes the proof of Case 2.  

We are now ready to conclude our proof. Let $F$ be a nonzero nonunit element of $K[S]$ such that $F$ has an infinite set  $\{F_i\}_{i\in I}$  of nonassociate divisors in $K[S]$. Set $n=\lcm \lbrace d(\supp F) \rbrace $.  Since $F(X^n)\in K[X]$ satisfies the FF-property in $K[S]$,  there exists an infinite subset $J$ of $I$ such that the elements of $\{F_i(X^n)\}_{i\in J}$ are associate in $K[S]$. Hence, the elements of $\{F_i\}_{i\in J}$ are associate divisors of $F$ in $K[S]$, which is a contradiction.
\end{proof}

\begin{rk}
Let $K$ be a field of characteristic $0$ and let $S$ be a Puiseux monoid. Assume that $S$ is an FFM.  We remark that if $f\in K[S]$ such that the roots of the polynomial  $f(X^m)$, where $m=\lcm \lbrace d\left( \supp f \right) \rbrace $, are elements of $ \{0\}\cup\{\zeta\in \overline{K}: \zeta^n=1, \;n\in \N\}$, then it follows,  by a similar proof as in Case 1 from the previous proof,  that $f$ satisfies the FF-property in  $K[S]$. In particular, the product of a monomial and a generalized cyclotomic polynomial has the FF-property in $K[S]$.  This generalizes Corollary \ref{C3}.
\end{rk}

We conclude this paper with the following corollary.
\begin{Co}
Let $D$ be an integral domain with quotient field $K$ and let $S$ be a Puiseux monoid. Assume that $K$ is finitely generated of characteristic $0$. Then the monoid domain $D[S]$ is an FFD if and only if $D$ is an FFD and $S$ is an FFM.
\end{Co} 
\begin{proof}
Follows from \cite[Propositions 1.4 and 1.5]{Kim} and Theorem \ref{L33}.
\end{proof}

\end{document}